\documentclass[a4paper, 12pt]{article}
\usepackage[margin=25mm]{geometry}
\usepackage{amsmath}
\usepackage{amssymb}
\usepackage{amsthm}
\usepackage{bbm}
\usepackage{titlesec}
\usepackage{titlefoot}
\usepackage{comment} 
\titleformat*{\section}{\normalsize\bfseries}
\allowdisplaybreaks
\def\R{\mathbb{R}}
\def\e{{\varepsilon}}        

\def\p{\partial}
\newtheorem{thm}{Theorem}[section]
\newtheorem{lem}[thm]{Lemma}
\newtheorem{cor}[thm]{Corollary}
\newtheorem{prop}[thm]{Proposition}
\newtheorem{rem}[thm]{Remark}

\makeatletter
\@addtoreset{equation}{section}

\makeatother

\begin{document}

\title{
\vspace{-1cm}
\large{Higher-order Asymptotic Profiles for Solutions to the Cauchy Problem for a Dispersive-dissipative Equation with a Cubic Nonlinearity}}
\author{Ikki Fukuda and Yota Irino}
\date{}
\maketitle


\vspace{-0.75cm}
\begin{abstract}
We consider the asymptotic behavior of solutions to the Cauchy problem for a dispersive-dissipative equation with a cubic nonlinearity. 
It is known that the leading term of the asymptotic profile for the solution to this problem is the Gaussian.
Moreover, by analyzing the corresponding integral equation, the higher-order asymptotic expansion for the solution to the linear part and the first asymptotic profile for the Duhamel term have already been obtained. 
In this paper, we construct the second asymptotic profile for the Duhamel term and give the more detailed higher-order asymptotic expansion of the solutions, 
which generalizes the previous works. Furthermore, we emphasize that the newly obtained higher-order asymptotic profiles have a good structure in the sense of satisfying the parabolic self-similarity.
\end{abstract}




\section{Introduction}

We consider the following Cauchy problem for a dispersive-dissipative equation:
\begin{align}
    \label{KdVB}
    \begin{split}
        &u_{t} - u_{xx} - D_{x}^{\alpha}\p_{x}u + \beta u^{2}u_{x} = 0, \ \ x \in \R, \ t>0, \\
        &u(x, 0) =u_{0}(x) , \ \ x \in \R,
    \end{split}
\end{align}
where $u=u(x, t)$ is a real-valued unknown function, $u_{0}(x)$ is a given initial data, $1<\alpha<3$ and $\beta \in \R$.
The subscripts $t$ and $x$ denote the partial derivatives with respect to $t$ and $x$, respectively. 
On the other hand, $D^{\alpha}_{x}$ stands for the fractional derivative defined via the Fourier transform by the formula 
\[D^{\alpha}_{x}f(x):=\mathcal{F}^{-1}\left[ |\xi|^{\alpha}\hat{f}(\xi) \right](x). \]
The definition of the Fourier transform $\hat{f}(\xi)=\mathcal{F}[f](\xi)$ will be given at the end of this section. 
The purpose of our study is to analyze the large time asymptotic behavior of the solutions to \eqref{KdVB}. 
Especially, we would like to investigate the higher-order asymptotic profiles of the solutions. 

First of all, we would like to introduce some previous results related to this problem. 
In order to do that, let us consider the following more generalized problem for \eqref{KdVB}: 
\begin{align}
    \label{gKdVB}
    \begin{split}
        &u_{t} - u_{xx} - D_{x}^{\alpha}\p_{x}u + \p_{x}(u^{q}) = 0, \ \ x \in \R, \ t>0, \\
        &u(x, 0) =u_{0}(x) , \ \ x \in \R,
    \end{split}
\end{align}
where $1<\alpha<3$ and $q\ge2$. Here, the nonlinear term $u^{q}$ should be interpreted either as 
\[
|u|^{q} \quad \text{or} \quad |u|^{q-1}u, 
\]
for negative $u$ and for non-integer $q$. We note that \eqref{gKdVB} becomes \eqref{KdVB} for $\beta=3$ in the case of $q=3$. 
In what follows, let us recall some known results on the Cauchy problem \eqref{gKdVB} for $1<\alpha<3$ and $q\ge2$. 
First, we shall introduce the basic results on the global existence and the decay estimates for the solutions to \eqref{gKdVB} given by Karch \cite{K99-1}. For the proofs of the following results, see Propositions 4.1 and 4.2 and Lemma 5.1 in \cite{K99-1}.   
\begin{prop}[\cite{K99-1}]\label{prop.GE-decay}
    Let $1<\alpha<3$, $q\ge2$ and $u_{0} \in H^{1}(\R)$. 
    Then, there exist $T>0$ and a unique local mild solution $u \in C([0, T); H^{1}(\R))$ to \eqref{gKdVB}. 
    Moreover, if $u_{0} \in H^{1}(\R)\cap L^{1}(\R)$ and $\| u_{0} \|_{H^{1}} + \| u_0 \|_{L^{1}}$ is sufficiently small, 
    then, there exists a unique global mild solution to \eqref{gKdVB} satisfying 
    \begin{equation}\label{GE}
    u\in L^{\infty}([0, \infty); L^{\infty}(\R)), \ \ t^{\frac{1}{4}}u\in L^{\infty}((0, \infty); L^{2}(\R)). 
    \end{equation}
     Furthermore, the solution $u(x, t)$ satisfies the following estimates:
        \begin{equation} 
        \left\| u(t) \right\|_{L^{p}} \le C(1 + t)^{-\frac{1}{2}\left(1-\frac{1}{p}\right)}, \ \  t\ge0, \ \ 
        \left\| u(t) \right\|_{L^{1}} \le C\left(1+t^{-\frac{\alpha-1}{4}}\right), \ \ t>0, \label{u-decay-Lp}
    \end{equation}
    for any $2\le p\le \infty$. In addition, the following estimate holds: 
    \begin{equation} \label{u-decay-dx}
        \left\| \p_{x} u(t) \right\|_{L^{2}} \le Ct^{-\frac{3}{4}}, \ \ t>0. 
    \end{equation}
    \end{prop}

The above result tells us that the solution $u(x, t)$ to \eqref{KdVB} decays at the same rate as the solutions to the parabolic equations such as the linear heat equation. Moreover, it is known that not only the decay rate but also the asymptotic behavior of the solution has some similar properties to those of the parabolic equations if $\alpha>1$. 
Next, we would like to explain about the detailed asymptotic behavior of the solution to \eqref{gKdVB} for $\alpha>1$. 
The asymptotic profile of the solution strongly depends on the nonlinear exponent $q$. 
In particular, $q=2$ and $q=3$ are the special cases in some sense.
In what follows, we start with introducing the known results of \eqref{gKdVB} for $q=2$. 
In this situation, the case of $\alpha=2$ has been significantly studied. 
When $q=2$ and $\alpha=2$, \eqref{gKdVB} becomes the well-known KdV--Burgers equation:
\begin{align}
    \label{KdVB-0}
    \begin{split}
        &u_{t} - u_{xx} + u_{xxx} + \p_{x}(u^{2}) = 0, \ \ x \in \R, \ t>0, \\
        &u(x, 0) =u_{0}(x) , \ \ x \in \R. 
    \end{split}
\end{align}
For the large time behavior of the solution to \eqref{KdVB-0}, it was first studied by Amick--Bona--Schonbek \cite{ABS89}. 
They derived the time decay estimates of the solution to \eqref{KdVB-0}. 
In particular, they showed that if $u_{0}\in H^{2}(\R) \cap L^{1}(\R)$, then the solution satisfies the first $L^{p}$-decay estimate in \eqref{u-decay-Lp} for any $2\le p \le \infty$. 
Moreover, Karch \cite{K99-2} studied \eqref{KdVB-0} in more details and extended the results given in \cite{ABS89}. 
Actually, if $u_{0}\in H^{1}(\R)\cap L^{1}(\R)$, then the following asymptotic formula holds: 
\begin{equation}\label{Karch}
\lim_{t\to \infty}t^{\frac{1}{2}\left(1-\frac{1}{p}\right)}\left\|u(t)-\chi(t)\right\|_{L^{p}}
=0,
\end{equation}
for any $1\le p\le \infty$, where $\chi(x, t)$ is the self-similar solution to the Burgers equation: 
\begin{equation*}
\chi_{t} - \chi_{xx} + \p_{x}(\chi^{2}) =0, \ \ x\in \R, \ t>0. 
\end{equation*}
We note that the self-similar solution $\chi(x, t)$ can be obtained explicitly (cf.~\cite{C51, H50}). 
Moreover, Hayashi--Naumkin \cite{HN06} improved the asymptotic rate given in \eqref{Karch} for $p=\infty$. 
After that, Kaikina--Ruiz-Paredes \cite{KP05} succeeded to construct the second asymptotic profile for the solution, under the additional condition $xu_{0}\in L^{1}(\R)$. In view of the second asymptotic profile, they proved that the optimal asymptotic rate to $\chi(x, t)$ is $t^{-1}\log t$ in $L^{\infty}(\R)$. For more related results to this problem \eqref{KdVB-0} in the similar situations, we can also refer to \cite{F19-1, HKNS06}. 

Now, we remark that this case $q=2$ in \eqref{gKdVB}, is a special case in the sense that the shape of the asymptotic profile changes compared with the case of $q>2$ in \eqref{gKdVB} (we will explain about the case of $q>2$ in the next paragraph). 
Even if the case of $\alpha=1$, the situation is similar to $\alpha=2$. When $\alpha=1$, the fractional dispersion term in \eqref{KdVB} can be rewritten by  
\[D^{1}_{x}\p_{x}u=\mathcal{F}^{-1}[i|\xi|\xi\hat{u}(\xi)](x)=\mathcal{F}^{-1}[-i\mathrm{sgn}(\xi)\widehat{u_{xx}}(\xi)](x)
=\mathcal{H}u_{xx}.\]
Therefore, \eqref{KdVB} transforms into the generalized Benjamin--Ono--Burgers equation: 
\begin{align}
    \label{BOB}
    \begin{split}
        &u_{t}-u_{xx}-\mathcal{H}u_{xx} + \p_{x}(u^{q}) = 0, \ \ x \in \R, \ t>0, \\
        &u(x, 0) =u_{0}(x) , \ \ x \in \R. 
    \end{split}
\end{align}
Here, the above operator $\mathcal{H}$ is called the Hilbert transform. In Dix \cite{D91}, he studied \eqref{BOB} for $q=2$ and prove that the solution $u(x, t)$ tends to the self-similar solution to \eqref{BOB} with $q=2$, 
under some suitable assumptions (see, also \cite{HKNS06}). However, compared with the case of $q=2$, for $q>2$, it is known that the asymptotic profile of the solution is different from the self-similar solution. 
Actually, Bona--Luo \cite{BL11} showed that if $q\ge3$ with $p\in \mathbb{N}$, the asymptotic profile of the solution $u(x, t)$ to \eqref{BOB} is exactly the same as that of the solution to the corresponding linearized equation. From the above perspective, one can see that the case $q=2$ is interesting not only in the case of $\alpha=1$ or $\alpha=2$ but also in the case of $\alpha$ is fractional. 
However, we note that for the case of $\alpha \neq1$ and $\alpha\neq2$ in \eqref{gKdVB}, the asymptotic profile of the solution has not yet been obtained for $q=2$, because the fractional dispersion term $D^{\alpha}_{x}\p_{x}u$ is difficult to treat. Although, this problem is very important issue, we do not deal with it in this paper. 
In our study, we would like to treat the other distinguished case $q=3$ in \eqref{gKdVB}, which will be explained in the next paragraph. In addition, note that the case of $\alpha=1$ is not treated too, because we focus on the case of $\alpha>1$, where the solution $u(x, t)$ behaves like the solution to the parabolic equations. 

In what follows, we shall introduce some known results for \eqref{gKdVB} in the case of $q>2$ and the more general $\alpha>1$. 
When $q>2$ in \eqref{gKdVB}, we can say that the nonlinearity is weak compared with the case of $q=2$, because if the solution $u(x, t)$ decays, 
then the nonlinear term $\p_{x}\left(u^{q}\right)$ decays faster than $-u_{xx}$. 
Because of this, the asymptotic profile of the solution changes essentially. 
Actually, Karch~\cite{K99-1} proved that the leading term of the solution $u(x, t)$ is governed by the solution to the linear heat equation. 
Moreover, in \cite{K99-1}, for the linear part of the solution to \eqref{gKdVB}, i.e. $\left(S_{\alpha}(t) *u_{0}\right)(x)$ in \eqref{gKdVB-IE} below, 
the higher-order asymptotic expansion has been obtained when $\alpha>1$. 
Especially, it is known that its asymptotic expansion strongly depends on the effect of the dispersion, i.e. the exponent $\alpha$. 
In order to explain such results, let us define the following Green function: 
\begin{equation}\label{def.S}
    S_{\alpha}(x, t) := \frac{1}{\sqrt{2\pi}} \mathcal{F}^{-1} \left[ e^{-t\xi^2+it|\xi|^{\alpha}\xi} \right](x).
\end{equation}
Then, from the Duhamel principle, we obtain the following integral equation associated with the Cauchy problem \eqref{gKdVB}:
\begin{align} \label{gKdVB-IE}
    u(t) = S_{\alpha}(t) *u_{0} - \int_{0}^{t} \p_{x} S_{\alpha}(t - \tau) * u^{q} (\tau) d\tau.
\end{align}
Now, we shall define $G(x, t)$ called the heat kernel and the constants $M$ and $m$ as follows:  
\begin{equation} \label{def.G-M}
G(x, t) := \frac{1}{\sqrt{4 \pi t}} e^{-\frac{x^2}{4t}}, \ \ M:=\int_{\R} u_0(x) dx, \ \ m:=\int_{\R}xu_{0}(x)dx. 
\end{equation}
Then, we first introduce the following asymptotic formulas for $\left(S_{\alpha}(t) *u_{0}\right)(x)$: 
\begin{thm}[\cite{K99-1}]\label{thm.Karch-L}
    Assume $u_{0} \in L^{1}(\R)$ and $xu_{0} \in L^{1}(\R)$. Then, the following asymptotics holds: 
        
\smallskip
\noindent
{\rm (i)}
Suppose $2<\alpha<3$. Then, for any $1\le p\le \infty$, we have 
     \begin{equation}\label{Karch-L1}
            \lim_{t \to \infty} t^{ \frac{1}{2}\left(1-\frac{1}{p}\right) + \frac{1}{2} } \left\| S_{\alpha}(t)*u_{0} -MG(t) + m\p_{x}G(t) \right\|_{L^{p}} = 0.
       \end{equation}
{\rm (ii)}
Suppose $\displaystyle \alpha=\frac{N+1}{N}$, where $N\in \mathbb{N}$. Then, for any $1\le p\le \infty$, we have 
     \begin{align}\label{Karch-L2}
     \begin{split}
            \lim_{t \to \infty} t^{ \frac{1}{2}\left(1-\frac{1}{p}\right) + \frac{1}{2} } \biggl\| S_{\alpha}(t)*u_{0} &- \sum_{k=0}^{N-1}\frac{t^{k}}{k!}(D_{x}^{\alpha}\p_{x})^{k}\left\{MG(t)-m\p_{x}G(t)\right\} \\
             &- \frac{M}{N!}(tD_{x}^{\alpha}\p_{x})^{N} G(t) \biggl\|_{L^{p}} = 0. 
             \end{split}
         \end{align}     
{\rm (iii)}
Suppose $\displaystyle \frac{N+2}{N+1}<\alpha<\frac{N+1}{N}$, where $N\in \mathbb{N}$. Then, for any $1\le p\le \infty$, we have 
     \begin{equation}\label{Karch-L3}
            \lim_{t \to \infty} t^{ \frac{1}{2}\left(1-\frac{1}{p}\right) + \frac{1}{2} } \left\| S_{\alpha}(t)*u_{0} - \sum_{k=0}^{N}\frac{t^{k}}{k!}(D_{x}^{\alpha}\p_{x})^{k}\left\{MG(t)-m\p_{x}G(t)\right\}  \right\|_{L^{p}} = 0.
        \end{equation}
Here, $S_{\alpha}(x, t)$ is defined by \eqref{def.S}, while $G(x, t)$, $M$ and $m$ are defined by \eqref{def.G-M}. 
\end{thm}

We note that for any of \eqref{Karch-L1}, \eqref{Karch-L2} and \eqref{Karch-L3}, the first term of their asymptotic expansions is given by $G(x, t)$. 
In addition, we can say that the effect of the dispersion term appears after the second terms of the expansions in \eqref{Karch-L2} and \eqref{Karch-L3}. 
On the other hand, dispersion term does not affect on \eqref{Karch-L1}.  

By virtue of Theorem~\ref{thm.Karch-L}, if we can get the asymptotic profile for the Duhamel term in \eqref{gKdVB-IE}, 
then we are able to obtain the asymptotic expansion for the solution $u(x, t)$ to \eqref{gKdVB}. 
Indeed, Karch \cite{K99-1} succeeded to construct the asymptotic profile for the Duhamel term in \eqref{gKdVB-IE}, if $q>2$, $u_{0}\in H^{1}(\R)\cap L^{1}(\R)$ and $xu_{0}\in L^{1}(\R)$.  
According to his results, its asymptotic profile can be divided into three cases: $2<q<3$; $q=3$; $q>3$. 
In this paper, we only focus on the case of $q=3$, which is a distinguished situation in which the asymptotic profile for the Duhamel term is linear but is multiplied by a logarithmic function $\log t$. 
More precisely, the following asymptotic formula has been proven in \cite{K99-1}: 
\begin{thm}[\cite{K99-1}]\label{thm.Karch-N}
    Let $1<\alpha<3$ and $q=3$. Assume that $u_{0} \in H^{1}(\R) \cap L^{1}(\R)$, $xu_{0} \in L^{1}(\R)$ and $\| u_{0} \|_{H^{1}} + \| u_{0} \|_{L^{1}}$ is sufficiently small.
    Suppose that $u(x, t)$ be the global mild solution to \eqref{gKdVB} mentioned in Proposition~\ref{prop.GE-decay}.
    Then, $u(x, t)$ satisfies the following asymptotics: 
\begin{equation}\label{Karch-N}
\lim_{t\to \infty}\frac{t^{\frac{1}{2}\left(1-\frac{1}{p}\right) + \frac{1}{2} }}{\log t}\left\|u(t)-S_{\alpha}(t)*u_{0}+\frac{M^{3}}{4\sqrt{3}\pi} (\log t)\p_{x} G(t)\right\|_{L^{p}}=0, 
\end{equation}
for any $1\le p\le \infty$, where $S_{\alpha}(x, t)$ is defined by \eqref{def.S}, while $G(x, t)$ and $M$ are defined by \eqref{def.G-M}. 
\end{thm}
Combining the above Theorems~\ref{thm.Karch-L} and \ref{thm.Karch-N}, 
we can get the higher-order asymptotic expansion for the solution $u(x, t)$ to \eqref{gKdVB} when $q=3$. 
However, the asymptotic rate $o(t^{-\frac{1}{2}\left(1-\frac{1}{p}\right) -\frac{1}{2} }\log t)$ given in \eqref{Karch-N} is slower than \eqref{Karch-L1}, \eqref{Karch-L2} and \eqref{Karch-L3} by $o(\log t)$. Therefore, the asymptotic rates related to the higher-order asymptotic expansions for the solution $u(x, t)$ also become slower than \eqref{Karch-L1}, \eqref{Karch-L2} and \eqref{Karch-L3}. 
This gives us the natural question whether it is possible to improve the asymptotic rate given in \eqref{Karch-N}.
For improving this asymptotic rate, it might be effective to investigate the second asymptotic profile for the Duhamel term in \eqref{gKdVB-IE}. 
In this paper, we analyze \eqref{KdVB}, which is \eqref{gKdVB} for $q=3$ with a parameter $\beta \in \R$. 
As our main results, we succeeded to derive the second asymptotic profile for the Duhamel term of the integral equation \eqref{KdVB-IE} below and improved the results given by Karch \cite{K99-1}. 
For the nonlinear parabolic equations such as \eqref{CD} and \eqref{SC} mentioned below, the higher-order asymptotic expansions of the solutions have been obtained in \cite{I-K} and \cite{KM09, Y09}, respectively. Therefore, the method used in these papers may be applied to our target problem \eqref{KdVB}. However, in the present paper, instead of using that technique, we apply a different approach to derive a kind of new type of asymptotic profile (see also Remark \ref{rem.self} below).

\medskip
This paper is organized as follows. In Section 2, 
we state our main results Theorem~\ref{thm.Duhamel} and Corollaries \ref{thm.Main} and \ref{cor.main} below, which generalize the results given in \cite{K99-1}. 
In order to prove them, we prepare a couple of lemmas and propositions in Section 3. 
Finally, we give the proof of our main results in Section 4. 
The main novelty of this paper is the derivation of the second asymptotic profile for the Duhamel term in \eqref{KdVB-IE} below. The key of this derivation is Proposition~\ref{prop.v-V-Psi} below, and the main techniques of its proof is based on the idea used in \cite{FI21} for the asymptotic analysis of the other dispersive-dissipative equation.  

\medskip
\par\noindent
\textbf{\bf{Notations.}} 

\smallskip
We denote the Fourier transform of $f$ and the inverse Fourier transform of $g$ as follows:
\begin{align*}
\hat{f}(\xi)=\mathcal{F}[f](\xi):=\frac{1}{\sqrt{2\pi}}\int_{\R}e^{-ix\xi}f(x)dx, \ \ \mathcal{F}^{-1}[g](x):=\frac{1}{\sqrt{2\pi}}\int_{\R}e^{ix\xi}g(\xi)d\xi.
\end{align*}

For $1\le p \le \infty$, $L^{p}(\R)$ means the usual Lebesgue spaces. 
Also, for $k\in \mathbb{N}\cup \{0\}$, we define the Sobolev spaces $H^{k}(\R)$ as the space of functions $u=u(x)$ such that $\p_{x}^{l}u$ are $L^{2}$-functions on $\R$ for $0\le l\le k$.  

Let $I\subseteq [0, \infty)$ be an interval and $X$ be a Banach space. Then, $L^{\infty}(I; X)$ denotes the space of all measurable functions $u: I \to X$ such that $\|u(t)\|_{X}$ belongs to $L^{\infty}(I)$. Also, $C(I; X)$ denotes the subspace of $L^{\infty}(I; X)$ of all continuous functions $u: I \to X$. 

Throughout this paper, $C$ denotes various positive constants, which may vary from line to line during computations. 
Also, it may depend on the norm of the initial data $u_{0}(x)$ and other parameters such as $\alpha$ and $\beta$. 
However, we note that it does not depend on the space variable $x$ and the time variable $t$. 
    
\section{Main Results}
    
In this section, we would like to state our main results. First of all, we shall rewrite the Cauchy problem \eqref{KdVB} to the integral form. It follows from the Duhamel principle that 
\begin{align}\label{KdVB-IE}
    u(x, t) = \left( S_{\alpha}(t) *u_{0} \right)(x) + I_{\alpha, \beta}[u](x, t), 
\end{align}
where $S_{\alpha}(x, t)$ is defined by \eqref{def.S}, while $I_{\alpha, \beta}[u](x, t)$ is defined by 
\begin{equation}\label{def.I}
I_{\alpha, \beta}[u](x, t):=-\frac{\beta}{3}\int_{0}^{t}\p_{x}S_{\alpha}(t-\tau)*u^{3}(\tau)d\tau.  
\end{equation}
For the global existence and the decay estimates of the solutions to \eqref{KdVB}, we can easily get the same results of Proposition \ref{prop.GE-decay}. 
Now, let us introduce the new function $\Psi(x, t)$ by 
    \begin{align}
        \Psi(x, t) &:= t^{-1} \Psi_{*} \left( \frac{x}{\sqrt{t}} \right), \ \ 
        \Psi_{*}(x) := \frac{d}{dx} \left( \int_{0}^{1} \left( G(1 - s) \ast F(s) \right) (x) ds \right), \label{def.Psi} \\
        F(y, s) &:= s^{-\frac{3}{2}} F_{*}\left( \frac{y}{\sqrt{s}} \right), \ \ 
        F_{*}(y) = \frac{1}{8\sqrt{\pi^3}} e^{-\frac{3y^2}{4}} - \frac{1}{8\sqrt{3\pi^3}} e^{-\frac{y^2}{4}}. \label{def.F}
    \end{align}
Moreover, we shall define the constant $\mathcal{M}$ by 
    \begin{equation}\label{def.m-mathM} 
        \mathcal{M} := \int_{0}^{1}\int_{\R}u^{3}(y, \tau)dyd\tau + \int_{1}^{\infty}\int_{\R} \left(u^{3} - \left(MG\right)^{3} \right)(y, \tau) dyd\tau.  
    \end{equation}
Then, we are able to obtain the second asymptotic profile for the Duhamel term $I_{\alpha, \beta}[u](x, t)$ in \eqref{KdVB-IE} defined by \eqref{def.I}, 
which generalizes Theorem~\ref{thm.Karch-N} mentioned in Section 1: 
\begin{thm}\label{thm.Duhamel}
    Let $1<\alpha<3$ and $\beta \in \R$. Assume that $u_{0} \in H^{1}(\R) \cap L^{1}(\R)$, $xu_{0} \in L^{1}(\R)$ and $\| u_{0} \|_{H^{1}} + \| u_{0} \|_{L^{1}}$ is sufficiently small.
    Then, \eqref{KdVB} has a unique global mild solution $u(x, t)$ satisfying the all properties mentioned in Proposition~\ref{prop.GE-decay}.
    Moreover, $u(x, t)$ satisfies the following asymptotics: 
    \begin{align}\label{I-ap}
    \begin{split}
   \lim_{t\to \infty} t^{\frac{1}{2}\left(1 - \frac{1}{p} \right) + \frac{1}{2}} \biggl\| I_{\alpha, \beta}[u](t) &+\frac{\beta M^3}{12\sqrt{3}\pi} (\log t)\p_{x} G(t) \\
   &+\frac{\beta \mathcal{M}}{3}\p_{x}G(t) + \frac{\beta M^{3}}{3}\Psi(t) \biggl\|_{L^{p}}=0,  
   \end{split}
    \end{align}
    for any $1 \le p \le \infty$, where $I_{\alpha, \beta}[u](x, t)$ is defined by \eqref{def.I}. Also, $G(x, t)$ and $M$ are defined by \eqref{def.G-M}, while $\Psi(x, t)$ and $\mathcal{M}$ are defined by \eqref{def.Psi} and \eqref{def.m-mathM}, respectively. 
\end{thm}
\begin{rem}\label{rem.self}
Compared with the previous works such as {\rm \cite{I-K, KM09, Y09}}, related to the higher-order asymptotic expansion of the solutions to the nonlinear parabolic equations, our formula is a different from them with respect to a structure of the asymptotic profile. Formally, one can derive the asymptotic profile of the Duhamel term, by substituting the heat kernel into $S_{\alpha}(x, t-\tau)$ and $u^{3}(x, \tau)$ in \eqref{def.I}: 
\[
I_{\alpha, \beta}[u](x, t) \sim -\frac{\beta M^{3}}{3}\int_{0}^{t}\p_{x} G(t-\tau)*G^{3}(\tau)d\tau+etc...
\]
The above integral has a singularity of $G^{3}(x, \tau)$ as $\tau \to 0$. 
Therefore, in order to control this singularity, $G^{3}(x, \tau)$ is usually replaced by $G^{3}(x, 1+\tau)$ as follows: 
\[
-\frac{\beta M^{3}}{3}\int_{0}^{t}\p_{x} G(t-\tau)*G^{3}(1+\tau)d\tau. 
\]
For details, see {\rm \cite{I-K, KM09, Y09}}. However, due to this change, the above asymptotic profile lost good structure with respect to the scaling. On the other hand, we emphasize that our newly obtained asymptotic profile $\Psi(x, t)$ has the parabolic self-similarity. More precisely, it satisfies  
\[
\Psi(x, t)=\lambda^{2}\Psi(\lambda x, \lambda^{2}t), \ \ \text{for} \ \ \lambda>0.
\]
\end{rem}
\begin{rem}
Under the assumptions in Theorem~\ref{thm.Duhamel}, the solution $u(x, t)$ to \eqref{KdVB} satisfies
   \begin{equation}\label{I-ap-cor}
   \left\| I_{\alpha, \beta}[u](t) +\frac{\beta M^3}{12\sqrt{3}\pi} (\log t)\p_{x} G(t) \right\|_{L^{p}}=\left(C_{*}+o(1)\right)t^{-\frac{1}{2}\left(1 - \frac{1}{p} \right) - \frac{1}{2}}
    \end{equation}
    as $t\to \infty$, for any $1 \le p \le \infty$, where the constant $C_{*}$ is defined by 
    \[
    C_{*}:= \left\| \frac{\beta \mathcal{M}}{3} \p_{x}G(1) + \frac{\beta M^{3}}{3}\Psi_{*}\right\|_{L^{p}}. 
    \]
This means that the asymptotic rate given in \eqref{Karch-N} can be improved, i.e. our result \eqref{I-ap-cor} generalizes the result \eqref{Karch-N} given in {\rm \cite{K99-1}}. 
\end{rem}
\begin{rem}
The assumption $\alpha>1$ guarantees that $S_{\alpha}(x, t)$ is well approximated by $G(x, t)$ as $t\to \infty$. For details, see Lemma~\ref{lem.S-ap} below. 
On the other hand, the condition $\alpha<3$ is a technical assumption. 
It will be used to prove Propositions \ref{prop.u-ap} and \ref{prop.S-G-error} below.
\end{rem}

Combining \eqref{KdVB-IE}, Theorems~\ref{thm.Duhamel} and \ref{thm.Karch-L}, 
we can get the following higher-order asymptotic expansion for the nonlinear solution $u(x, t)$ to \eqref{KdVB}:
\begin{cor}\label{thm.Main}
    Under the same assumptions in Theorem~\ref{thm.Duhamel}, the solution $u(x, t)$ to \eqref{KdVB} satisfies the following asymptotics:
    
\smallskip
\noindent
{\rm (i)}
Suppose $2<\alpha<3$. Then, for any $1\le p\le \infty$, we have 
     \begin{align}
     \begin{split}\label{main-1}
            \lim_{t \to \infty} t^{ \frac{1}{2}\left(1-\frac{1}{p}\right) + \frac{1}{2} } \biggl\| u(t) &-MG(t) +\frac{\beta M^3}{12\sqrt{3}\pi} (\log t)\p_{x} G(t) \\
        &+ \left(m+\frac{\beta \mathcal{M}}{3}\right)\p_{x}G(t) + \frac{\beta M^{3}}{3}\Psi(t) \biggl\|_{L^{p}} = 0.
        \end{split}
       \end{align}
{\rm (ii)}
Suppose $\displaystyle \alpha=\frac{N+1}{N}$, where $N\in \mathbb{N}$. Then, for any $1\le p\le \infty$, we have 
     \begin{align}
     \begin{split}\label{main-2}
            \lim_{t \to \infty} t^{ \frac{1}{2}\left(1-\frac{1}{p}\right) + \frac{1}{2} } \biggl\| u(t) &- \sum_{k=0}^{N-1}\frac{t^{k}}{k!}(D_{x}^{\alpha}\p_{x})^{k}\left\{MG(t)-m\p_{x}G(t)\right\}  - \frac{M}{N!}(tD_{x}^{\alpha}\p_{x})^{N} G(t)  \\
        &+\frac{\beta M^3}{12\sqrt{3}\pi} (\log t)\p_{x} G(t) + \frac{\beta \mathcal{M}}{3}\p_{x}G(t) + \frac{\beta M^{3}}{3}\Psi(t) \biggl\|_{L^{p}} = 0.
         \end{split}
         \end{align}
{\rm (iii)}
Suppose $\displaystyle \frac{N+2}{N+1}<\alpha<\frac{N+1}{N}$, where $N\in \mathbb{N}$. Then, for any $1\le p\le \infty$, we have 
     \begin{align}
     \begin{split}\label{main-3}
            \lim_{t \to \infty} t^{ \frac{1}{2}\left(1-\frac{1}{p}\right) + \frac{1}{2} } \biggl\| u(t) &- \sum_{k=0}^{N}\frac{t^{k}}{k!}(D_{x}^{\alpha}\p_{x})^{k}\left\{MG(t)-m\p_{x}G(t)\right\} \\
        &+\frac{\beta M^3}{12\sqrt{3}\pi} (\log t)\p_{x} G(t) + \frac{\beta \mathcal{M}}{3}\p_{x}G(t) + \frac{\beta M^{3}}{3}\Psi(t) \biggl\|_{L^{p}} = 0.
        \end{split}
        \end{align}
Here, $G(x, t)$, $M$ and $m$ are defined by \eqref{def.G-M}, while $\Psi(x, t)$ and $\mathcal{M}$ are defined by \eqref{def.Psi} and \eqref{def.m-mathM}, respectively. 
\end{cor}

Moreover, by virtue of the above result, we can also obtain the following asymptotic formulas with the optimal decay order:
\begin{cor}\label{cor.main}
Under the same assumptions in Theorem~\ref{thm.Duhamel}, the following estimates hold true: 

\smallskip
\noindent
{\rm (i)}
Suppose $2<\alpha<3$. Then, we have 
\begin{equation}\label{cor-1}
\left\| u(t)-MG(t) + \frac{\beta M^3}{12\sqrt{3}\pi} (\log t)\p_{x} G(t) \right\|_{L^{p}}
=\left( C_{\dag} + o(1) \right) t^{ -\frac{1}{2}\left(1-\frac{1}{p}\right) - \frac{1}{2} }
\end{equation}
as $t\to \infty$, for any $1 \le p \le \infty$. 

\smallskip
\noindent
{\rm (ii)}
Suppose $\displaystyle \alpha=\frac{N+1}{N}$, where $N\in \mathbb{N}$. Then, we have 
\begin{equation}\label{cor-2}
\left\| u(t)-M\sum_{k=0}^{N-1}\frac{t^{k}}{k!}(D_{x}^{\alpha}\p_{x})^{k}G(t)+ \frac{\beta M^3}{12\sqrt{3}\pi} (\log t)\p_{x} G(t) \right\|_{L^{p}} 
=\left( C_{\dag} + o(1) \right) t^{ -\frac{1}{2}\left(1-\frac{1}{p}\right) - \frac{1}{2} } 
\end{equation}
as $t\to \infty$, for any $1 \le p \le \infty$. 

\smallskip
\noindent
{\rm (iii)}
Suppose $\displaystyle \frac{N+2}{N+1}<\alpha<\frac{N+1}{N}$, where $N\in \mathbb{N}$. Then, we have 
\begin{equation}\label{cor-3}
\left\| u(t)-M\sum_{k=0}^{N}\frac{t^{k}}{k!}(D_{x}^{\alpha}\p_{x})^{k}G(t)+ \frac{\beta M^3}{12\sqrt{3}\pi} (\log t)\p_{x} G(t) \right\|_{L^{p}}
=\left( C_{\dag} + o(1) \right) t^{ -\frac{1}{2}\left(1-\frac{1}{p}\right) - \frac{1}{2} }
\end{equation}
as $t\to \infty$, for any $1 \le p \le \infty$. 
Here, the constant $C_{\dag}$ is defined by 
\begin{equation}\label{const-d}
C_{\dag}:=
\begin{cases}
\displaystyle \left\| \left(m+ \frac{\beta \mathcal{M}}{3}\right)\p_{x}G(1) + \frac{\beta M^{3}}{3}\Psi_{*}\right\|_{L^{p}}, 
\ \ \text{for} \ \ {\rm (i)} \ \ \text{or} \ \ {\rm (iii)}, \\[1em]
\displaystyle \left\| \left(m+ \frac{\beta \mathcal{M}}{3}\right)\p_{x}G(1) + \frac{\beta M^{3}}{3}\Psi_{*}- \frac{M}{N!}(D_{x}^{\alpha}\p_{x})^{N} G(1)\right\|_{L^{p}}, 
\ \ \text{for} \ \ {\rm (ii)}. 
\end{cases}
\end{equation}
\end{cor}

\begin{rem}
We note that the above formulas \eqref{cor-1}, \eqref{cor-2} and \eqref{cor-3} include not only the upper bound but also the lower bound of the $L^{p}$-norm. 
Therefore, \eqref{cor-1}, \eqref{cor-2} and \eqref{cor-3} give us the second, $(N+1)$th and $(N+2)$th order complete asymptotic expansion of $u(x, t)$ with the optimal asymptotic rate $t^{-\frac{1}{2}\left(1-\frac{1}{p}\right) -\frac{1}{2} }$, respectively. 
Here, we note that $(tD_{x}^{\alpha}\p_{x})^{k}\p_{x}G(x, t)$ for $k=1, 2, \cdots, N-1$ in \eqref{main-2} and for $k=1, 2, \cdots, N$ in \eqref{main-3} have been disappeared because these terms decay faster than $t^{ -\frac{1}{2}\left(1-\frac{1}{p}\right) - \frac{1}{2} }$. In other words, they essentially do not affect asymptotic behavior up to the decay order $t^{ -\frac{1}{2}\left(1-\frac{1}{p}\right) - \frac{1}{2} }$. 
\end{rem}

Finally in the rest of this section, let us make some comments for other related equations. 
We remark that in the paper \cite{K99-1} mentioned above, the similar result to \eqref{Karch-N} has also been obtained for the generalized Benjamin--Bona--Mahony--Burgers (BBM--Burgers) equation: 
\[
u_{t} - u_{xx} -u_{xxt} + u_{xxx}  + \p_{x}(u^{q}) =0, \ \ x\in \R, \ t>0. 
\]
In addition, that result was extended to the two-dimensional case by Prado--Zuazua \cite{PZ01}. 
Moreover, some results were also observed for other nonlinear parabolic equations where the solutions tend to the heat kernel and the logarithmic term appears in the asymptotic behavior.
For example, in Zuazua \cite{Z93}, the corresponding result to \eqref{Karch-N} also can be established for the following convection-diffusion equation in arbitrary space dimension:   
\begin{equation}\label{CD}
u_{t} - \Delta u =a\cdot\nabla  (u^{q}), \ \  x\in\mathbb{R}^n, \ t>0, 
\end{equation}
where $q>1$ and $a\in \R^{n}$. We note that the same situation of \eqref{Karch-N} (i.e. logarithmic term appears) occurs at only the case of $q=1+\frac{2}{n}$. Furthermore, it is known that the solutions $(u, v)$ to the following parabolic system of chemotaxis also goes to the heat kernel $G(x, t)$: 
\begin{align}\label{SC}
\begin{split}
&u_{t}=\Delta u-\nabla \cdot (u\nabla v), \ \ x\in \R^{n}, \ t>0, \\
&v_{t}=\Delta v-v+u, \ \ x\in \R^{n}, \ t>0. 
\end{split}
\end{align}
Especially, in Nagai--Yamada \cite{NY07}, the related result to \eqref{Karch-N} has been derived, in the case of $n=1$. For details about the asymptotic analysis for the above equations, see e.g. \cite{HKNS06, I-K, K99-1, KM09, NY07, PZ01, Y09, Z93} and also references therein. 
We believe that the our method developed in this paper also can be applied to the above problems. 

\section{Preliminaries}

In this section, we prepare a couple of lemmas and propositions to prove the main results. 
First of all, let us introduce some decay properties for the Green function $S_{\alpha}(x, t)$ defined by \eqref{def.S}. 
This function has the similar decay estimates as the heat kernel $G(x, t)$ defined by \eqref{def.G-M} (see \eqref{G-est} below). 
In addition, $S_{\alpha}(x, t)$ is well approximated by $G(x, t)$ as $t \to \infty$ when $\alpha>1$. 
More precisely, the following estimate \eqref{S-ap-1} holds true. 
For these proofs, see Lemmas 3.1 and 3.2 in \cite{K99-1}. 
The following two lemmas will be used in the proofs of Propositions \ref{prop.u-ap} and \ref{S-G-error} below: 
\begin{lem} \label{lem.S-est}
    Let $\alpha>1$ and $l \in \mathbb{N}\cup\{0\}$. 
    Then, for any $2\le p\le \infty$, we have
    \begin{align}
        &\left\| \p_{x}^{l} S_{\alpha}(t) \right\|_{L^{p}} \le C t^{-\frac{1}{2} \left( 1 - \frac{1}{p} \right) - \frac{l}{2}}, \ \ t>0, \label{S-est} \\
        &\left\| \p_{x}^{l} S_{\alpha}(t) \right\|_{L^{1}} \le C t^{- \frac{l}{2}}\left(1 + t^{-\frac{\alpha-1}{4}}\right), \ \ t>0.  \label{S-est-1} 
     \end{align}
\end{lem}

\begin{lem} \label{lem.S-ap}
    Let $\alpha>1$ and $l \in \mathbb{N}\cup\{0\}$.  
    Then, for any $1\le p\le \infty$, we have
    \begin{align}
        \left\| \p_{x}^{l} \left( S_{\alpha}- G\right)(t)\right\|_{L^{p}} \le C t^{-\frac{1}{2}\left( 1 - \frac{1}{p} \right) - \frac{\alpha-1}{2} -\frac{l}{2}}, \ \ t>0. \label{S-ap-1} 
    \end{align}
\end{lem}

Next, we shall introduce the $L^{p}$-decay estimate for $G(x, t)$ defined by \eqref{def.G-M}. 
The following estimate \eqref{G-est} with $k=0$ is well known (see e.g. \cite{GGS99}). 
In what follows, we would like to introduce a slightly generalized estimate including the fractional derivatives $(D_{x}^{\alpha}\p_{x})^{k}$. 
Since $G(x,t)$ has the parabolic self-similarity, some may think that the following lemma is clear.
However, let us give its proof for the reader's convenience. 
\begin{lem} \label{lem.G-est}
    Let $\alpha>1$ and $k, l \in \mathbb{N}\cup\{0\}$.
    Then, for any $1\le p\le \infty$, we have
    \begin{equation}\label{G-est}
        \left\| (D_{x}^{\alpha}\p_{x})^{k}\p_{x}^{l} G(t) \right\|_{L^{p}} \le C t^{-\frac{1}{2}\left( 1 - \frac{1}{p} \right) - \frac{k(\alpha+1)}{2} - \frac{l}{2} }, \ \ t > 0. 
    \end{equation}
\end{lem}
\begin{proof}
First, it follows from the Fourier transform that 
\[
(D_{x}^{\alpha}\p_{x})^{k}\p_{x}^{l} G(x, t) =\frac{1}{2\pi}\int_{\R}\left\{|\xi|^{\alpha}(i\xi)\right\}^{k}(i\xi)^{l}e^{-t\xi^{2}+ix\xi}d\xi. 
\]
Therefore, we immediately obtain \eqref{G-est} for $p=\infty$ as follows: 
\begin{align}
\left\|(D_{x}^{\alpha}\p_{x})^{k}\p_{x}^{l}G(t)\right\|_{L^{\infty}}
\le \int_{\R}|\xi|^{k(\alpha+1)+l}e^{-t\xi^{2}}d\xi
\le Ct^{-\frac{1}{2} - \frac{k(\alpha+1)}{2} -\frac{l}{2} }, \ \ t>0. \label{G-frac-Linf}
\end{align}

Next, let us prove \eqref{G-est} for $p=1$. Here, in order to handle the $L^{1}$-norm, we shall use the following inequality (cf. \cite{K97-1, K97-2}): 
\begin{equation*}
\left\|\mathcal{F}^{-1}[g]\right\|_{L^{1}}\le C\left\|g\right\|_{L^{2}}^{\frac{1}{2}}\left\|\p_{\xi}g\right\|_{L^{2}}^{\frac{1}{2}}, \ \ g\in H^{1}(\R).
\end{equation*}
Now, applying this inequality for $g(\xi, t):=\frac{1}{\sqrt{2\pi}}\left\{|\xi|^{\alpha}(i\xi)\right\}^{k}(i\xi)^{l}e^{-t\xi^{2}}$, we get 
\begin{align}
&\left\|(D_{x}^{\alpha}\p_{x})^{k}\p_{x}^{l}G(t)\right\|_{L^{1}}
\le C\left\|g(t)\right\|_{L^{2}}^{\frac{1}{2}}\left\|\p_{\xi}g(t)\right\|_{L^{2}}^{\frac{1}{2}} \nonumber \\
&\le C\left(\int_{\R}\xi^{2k(\alpha+1)+2l}e^{-2t\xi^{2}}d\xi \right)^{\frac{1}{4}}
\left(\int_{\R}\left\{\xi^{2k(\alpha+1)-2+2l}+t^{2}\xi^{2k(\alpha+1)+2+2l}\right\}e^{-2t\xi^{2}}d\xi \right)^{\frac{1}{4}} \nonumber \\
&\le C\left(t^{-\frac{1}{2}-k(\alpha+1)-l}\right)^{\frac{1}{4}}\left\{ \left(t^{-\frac{1}{2}-k(\alpha+1)+1-l}\right)^{\frac{1}{4}} + t^{\frac{1}{2}} \left(t^{-\frac{1}{2}-k(\alpha+1)-1-l}\right)^{\frac{1}{4}} \right\} \nonumber \\
&\le Ct^{-\frac{k(\alpha+1)}{2}-\frac{l}{2}}, \ \ t>0. \label{G-frac-L1}
\end{align}
Thus, we obtain \eqref{G-est} from \eqref{G-frac-Linf} and \eqref{G-frac-L1} through the interpolation inequality: 
\begin{equation}\label{inter}
\left\|g\right\|_{L^{p}} \le \left\|g\right\|_{L^{\infty}}^{1-\frac{1}{p}}\left\|g\right\|_{L^{1}}^{\frac{1}{p}}, \ \ 1\le p\le \infty.
\end{equation}
This completes the proof. 
\end{proof}

Moreover, we would like to introduce the following asymptotic formula for the solution $(G(t)*u_{0})(x)$ to the linear heat equation (for the proof, see e.g. \cite{GGS99, K99-1}): 
\begin{lem} \label{lem.G-ap}
    Let $l \in \mathbb{N}\cup\{0\}$. Suppose $u_{0}\in L^{1}(\R)$ and $xu_{0}\in L^{1}(\R)$. 
        Then, for any $1\le p\le \infty$, we have
            \begin{equation}\label{G-ap}
        \left\| \p_{x}^{l} \left( G(t)*u_{0}-MG(t) \right) \right\|_{L^{p}} \le C t^{-\frac{1}{2}\left( 1 - \frac{1}{p} \right) - \frac{1}{2} - \frac{l}{2}}, \ \ t > 0, 
    \end{equation}
    where $G(x, t)$ and $M$ are defined by \eqref{def.G-M}.
\end{lem}

Finally, in the rest of this section, let us derive an important result which related to the approximation for the nonlinear term $u^{3}(x, t)$. Actually, the proof of Theorem \ref{thm.Duhamel} needs the following proposition: 
\begin{prop} \label{prop.u-ap}
    Let $1<\alpha<3$ and $\beta \in \R$. Assume that $u_{0} \in H^{1}(\R) \cap L^{1}(\R)$, $xu_{0} \in L^{1}(\R)$ and $\| u_0 \|_{H^{1}} + \| u_0 \|_{L^{1}}$ is sufficiently small.
        Then, the solution $u(x, t)$ to \eqref{KdVB} satisfies
     \begin{equation} \label{u^3-ap-est}
        \left\| u^{3}(t) - (M G)^{3}(t) \right\|_{L^{p}} \le Ct^{-\frac{1}{2}\left(1-\frac{1}{p}\right)-1} \left\{t^{-\frac{\min\{\alpha-1, 1\}}{2} }+t^{-\frac{1}{2}}\log(2 + t)\right\}, \ \ t\ge1,
    \end{equation}
    for any $1\le p \le \infty$, where $G(x, t)$ and $M$ are defined by \eqref{def.G-M}.
\end{prop}
\begin{proof}
    First, we shall prove that the solution $u(x, t)$ to \eqref{KdVB} satisfies
    \begin{equation} \label{u-ap-est}
        \left\| u(t) - M G(t) \right\|_{L^{p}} \le Ct^{-\frac{1}{2}\left(1-\frac{1}{p}\right)} \left\{t^{-\frac{\min\{\alpha-1, 1\}}{2} }+t^{-\frac{1}{2}}\log(2 + t)\right\}, \ \ t\ge1, 
    \end{equation}
    for any $1\le p \le \infty$. From \eqref{KdVB-IE}, we can see that the following relation holds: 
    \begin{equation}\label{u-MG}
            u(x, t) - M G(x, t) = \left(S_{\alpha}(t)*u_{0}\right)(x) - MG(x, t) + I_{\alpha, \beta}[u](x, t). 
    \end{equation}
   By virtue of Lemmas \ref{lem.S-ap} and \ref{lem.G-ap}, from the Young inequality, we obtain 
   \begin{align}
   &\left\|S_{\alpha}(t)*u_{0} - MG(t)\right\|_{L^{p}}  \nonumber \\
   & \le \left\|S_{\alpha}(t)-G(t)\right\|_{L^{p}}\left\|u_{0}\right\|_{L^{1}} + \left\|G(t)*u_{0}- MG(t)\right\|_{L^{p}} \nonumber \\
   &\le C t^{-\frac{1}{2}\left( 1 - \frac{1}{p} \right) - \frac{\alpha-1}{2}} + C t^{-\frac{1}{2}\left( 1 - \frac{1}{p} \right) - \frac{1}{2}}
   \le Ct^{-\frac{1}{2}\left(1-\frac{1}{p}\right) -\frac{\min\{\alpha-1, 1\}}{2} }, \ \ t\ge1, \label{u-ap-pre}
   \end{align}   
   for any $1\le p \le \infty$. Therefore, in order to prove \eqref{u-ap-est}, we only need to evaluate the Duhamel term $I_{\alpha, \beta}[u](x, t)$ in \eqref{u-MG}. In order to do that, let us prepare an auxiliary estimates for the nonlinear term $u^{3}(x, t)$. For any $1\le p \le \infty$, it follows from \eqref{u-decay-Lp} that 
    \begin{align}
            \left\| u^3(t) \right\|_{L^{p}} 
            \le \left\| u(t) \right\|_{L^{\infty}} \left\| u(t) \right\|_{L^{2p}}^{2} 
            \le C(1 + t)^{-\frac{1}{2}\left(1-\frac{1}{p}\right)-1}, \ \ t\ge0.  \label{u^3-est-Lp}
    \end{align}
    
Now, let us evaluate the Duhamel term $I_{\alpha, \beta}[u](x, t)$. 
From the Young inequality, \eqref{S-est}, \eqref{S-est-1} and \eqref{u^3-est-Lp}, we have  
    \begin{align}
            \left\|I_{\alpha, \beta}[u](t)\right\|_{L^{p}}  
            &\le C\int_{0}^{\frac{t}{2}} \left\| \p_{x}S_{\alpha}(t - \tau) \right\|_{L^{p}} \left\| u^{3}(\tau) \right\|_{L^{1}} d\tau  \nonumber\\
            &\quad + C\int_{\frac{t}{2}}^{t} \left\| \p_{x} S_{\alpha}(t - \tau) \right\|_{L^{1}} \left\| u^{3}(\tau) \right\|_{L^{p}} d\tau  \nonumber\\
            &\le C \int_{0}^{\frac{t}{2}} (t - \tau)^{- \frac{1}{2}\left(1-\frac{1}{p}\right)-\frac{1}{2}} \left(1+(t-\tau)^{-\frac{\alpha-1}{4}}\right) (1 + \tau)^{-1}d\tau \nonumber \\
            &\quad + C \int_{\frac{t}{2}}^{t} (t - \tau)^{- \frac{1}{2}} \left(1+(t-\tau)^{-\frac{\alpha-1}{4}}\right) (1 + \tau)^{- \frac{1}{2}\left(1-\frac{1}{p}\right)-1}d\tau  \nonumber\\
    &\le Ct^{- \frac{1}{2}\left(1-\frac{1}{p}\right)-\frac{1}{2}} \left(1+t^{-\frac{\alpha-1}{4}}\right) \log(2 + t) 
            + Ct^{- \frac{1}{2}\left(1-\frac{1}{p}\right)-1}\left(t^{\frac{1}{2}}+t^{\frac{1}{2}-\frac{\alpha-1}{4}}\right) \nonumber \\
            &\le Ct^{- \frac{1}{2}\left(1-\frac{1}{p}\right)-\frac{1}{2}}\log(2+t), \ \ t\ge1, \ 1<\alpha<3,    \label{Duhamel-Lp}
    \end{align}
for any $1\le p \le \infty$. Thus, combining \eqref{u-ap-pre} and \eqref{Duhamel-Lp}, we arrive at \eqref{u-ap-est}. 

Finally, we would like to prove \eqref{u^3-ap-est}. 
By using \eqref{u-ap-est}, \eqref{u-decay-Lp} and Lemma~\ref{lem.G-est}, we obtain
 \begin{align*} 
            &\left\| u^{3}(t)-(MG)^{3}(t)\right\|_{L^{p}}  \\
            & = \left\| \left\{\left(u-MG\right)\left(u^{2}+u(MG)+(MG)^{2}\right)\right\}(t) \right\|_{L^{p}} \\
            &\le C \left\| u(t)-MG(t)\right\|_{L^{p}} 
            \left( \left\| u(t) \right\|_{L^{\infty}}^{2} + \left\| u(t) \right\|_{L^{\infty}}\left\| G(t) \right\|_{L^{\infty}} +\left\| G(t) \right\|_{L^{\infty}}^{2}\right) \\
            &\le Ct^{-\frac{1}{2}\left(1-\frac{1}{p}\right)-1} \left\{t^{-\frac{\min\{\alpha-1, 1\}}{2} }+t^{-\frac{1}{2}}\log(2 + t)\right\}, \ \ t\ge1, 
  \end{align*}
for any $1\le p \le \infty$. Thus, we can say that the desired result \eqref{u^3-ap-est} is true.   
\end{proof}

\section{Proof of the Main Results}

In this section, we shall prove our main results Theorem \ref{thm.Duhamel} and Corollaries \ref{thm.Main} and \ref{cor.main}. 
In order to do that, for simplicity, we define the following functions: 
\begin{align} 
    &v(x, t) := \int_{1}^{t} \p_x G(t - \tau) * G^{3}(\tau) d\tau, \quad 
    V(x, t) := \frac{1}{4\sqrt{3}\pi} (\log t)\p_{x} G(x, t), \label{def.v-V}\\
    &W(x, t) :=\int_{0}^{1}\p_{x}G(t-\tau) * u^{3}(\tau) d\tau + \int_{1}^{t} \p_{x} G(t - \tau) * \left(u^{3} - (MG)^{3} \right)(\tau) d\tau.  \label{def.W}
\end{align}
Then, from \eqref{def.I}, \eqref{def.m-mathM}, \eqref{def.v-V} and \eqref{def.W}, we can see that the following relation holds: 
\begin{align} 
        &I_{\alpha, \beta}[u](x, t) +\frac{\beta M^3}{12\sqrt{3}\pi} (\log t)\p_{x} G(x, t) +\frac{\beta \mathcal{M}}{3}\p_{x}G(x, t) + \frac{\beta M^{3}}{3}\Psi(x, t) \nonumber\\ 
        &= - \frac{\beta}{3} \left\{ \int_{0}^{t} \p_{x} (S_{\alpha} - G)(t - \tau) * u^{3}(\tau) d\tau \right\} - \frac{\beta}{3} \left\{ W(x, t) - \mathcal{M} \p_{x} G(x, t) \right\} \nonumber \\
        &\quad \,-\frac{\beta M^{3}}{3}\left\{ v(x, t) - V(x, t) - \Psi(x, t) \right\}. \label{formulation}
\end{align}
In what follows, to complete the proof of Theorem \ref{thm.Duhamel}, we shall evaluate the all terms in the right hand side of \eqref{formulation}. 
First, let us prove the following decay estimate for the first term of the above:  
\begin{prop} \label{prop.S-G-error}
   Let $1<\alpha<3$ and $\beta \in \R$. Assume that $u_{0} \in H^{1}(\R) \cap L^{1}(\R)$, $xu_{0} \in L^{1}(\R)$ and $\| u_0 \|_{H^{1}} + \| u_0 \|_{L^{1}}$ is sufficiently small.
        Then, the solution $u(x, t)$ to \eqref{KdVB} satisfies
    \begin{equation}\label{S-G-error}
        \left\| \int_{0}^{t} \p_{x} (S_{\alpha} - G)(t - \tau) * u^{3}(\tau) d\tau \right\|_{L^{p}} \le Ct^{-\frac{1}{2}\left(1 - \frac{1}{p}\right) -\frac{\alpha}{2}} \log(2 + t), \ \ t\ge2, 
    \end{equation}
      for any $1\le p \le \infty$, where $S_{\alpha}(x, t)$ and $G(x, t)$ are defined by \eqref{def.S} and \eqref{def.G-M}, respectively. 
\end{prop}
\begin{proof}
Before proving \eqref{S-G-error}, we need to prepare some auxiliary decay estimates for $\p_{x}(u^{3}(x, t))$. 
First, it follows from \eqref{u-decay-Lp}, \eqref{u-decay-dx} and the Cauchy--Schwarz inequality that 
\begin{align}
\left\| \p_{x}\left(u^{3}(t)\right) \right\|_{L^{2}}
&\le 3\left\| u(t) \right\|_{L^{\infty}}^{2}\left\| u_{x}(t) \right\|_{L^{2}}
\le Ct^{-\frac{7}{4}}, \ \ t>0, \label{u^3-dx-L2}\\
\left\| \p_{x}\left(u^{3}(t)\right) \right\|_{L^{1}}
&\le 3\left\| u(t) \right\|_{L^{\infty}}\left\| u(t) \right\|_{L^{2}}\left\| u_{x}(t) \right\|_{L^{2}}
\le Ct^{-\frac{3}{2}}, \ \ t>0. \label{u^3-dx-L1}
\end{align}
Next, let us derive the $L^{\infty}$-norm of $\p_{x}(u^{3}(x, t))$. In order to do that, we shall evaluate the $L^{\infty}$-norm of $u_{x}(x, t)$. 
From \eqref{KdVB-IE}, \eqref{def.I}, the Young inequality, \eqref{S-est}, \eqref{u^3-est-Lp} and \eqref{u^3-dx-L2}, we obtain 
\begin{align}
\left\| u_{x}(t)\right\|_{L^{\infty}}
& \le \left\|\p_{x}S_{\alpha}(t)\right\|_{L^{\infty}}\left\|u_{0}\right\|_{L^{1}} 
 + C\int_{0}^{\frac{t}{2}} \left\| \p_{x}^{2}S_{\alpha}(t - \tau) \right\|_{L^{\infty}} \left\| u^{3}(\tau) \right\|_{L^{1}} d\tau  \nonumber \\
&\quad + C\int_{\frac{t}{2}}^{t} \left\| \p_{x}S_{\alpha}(t - \tau) \right\|_{L^{2}} \left\| \p_{x}\left(u^{3}(\tau)\right) \right\|_{L^{2}} d\tau \nonumber  \\
&\le Ct^{-1} + C\int_{0}^\frac{t}{2} (t - \tau)^{- \frac{3}{2}} (1 + \tau)^{-1} d\tau 
        + C\int_{\frac{t}{2}}^{t} (t-\tau)^{-\frac{3}{4}} \tau^{-\frac{7}{4}} d\tau \nonumber \\
        &\le Ct^{-1} + Ct^{-\frac{3}{2}}\log(2+t) + Ct^{-\frac{3}{2}} 
        \le Ct^{-1}, \ \ t\ge1. \label{u-dx-Linf}
\end{align}
Therefore, from \eqref{u-decay-Lp} and \eqref{u-dx-Linf}, we can easily derive that 
\begin{equation}\label{u^3-dx-Linf}
\left\| \p_{x}\left(u^{3}(t)\right) \right\|_{L^{\infty}}
\le 3\left\| u(t)\right\|_{L^{\infty}}^{2}\left\| u_{x}(t)\right\|_{L^{\infty}}
\le Ct^{-2}, \ \ t\ge1.
\end{equation}
Moreover, by virtue of \eqref{inter}, \eqref{u^3-dx-L1} and \eqref{u^3-dx-Linf}, for any $1\le p\le \infty$, we get 
\begin{equation}\label{u^3-dx-Lp}
\left\| \p_{x}\left(u^{3}(t)\right) \right\|_{L^{p}}\le Ct^{-\frac{1}{2}\left(1-\frac{1}{p}\right)-\frac{3}{2}}, \ \ t\ge1.
\end{equation}

Now, we would like to prove \eqref{S-G-error}. By using the Young inequality, \eqref{S-ap-1}, \eqref{u^3-est-Lp} and \eqref{u^3-dx-Lp}, we obtain 
    \begin{align}
        &\left\| \int_{0}^{t} \p_{x} (S_{\alpha} - G)(t - \tau) * u^{3}(\tau) d\tau \right\|_{L^{p}} \nonumber \\
        &\le \int_{0}^{\frac{t}{2}} \left\| \p_{x}(S_{\alpha} - G)(t - \tau) \right\|_{L^{p}} \left\| u^{3}(\tau) \right\|_{L^{1}} d\tau \nonumber \\
        &\quad + \int_{\frac{t}{2}}^{t} \left\| \left(S_{\alpha}-G\right)(t - \tau) \right\|_{L^{1}} \left\| \p_{x}\left(u^{3}(\tau)\right) \right\|_{L^{p}} d\tau \nonumber\\
        & \le C\int_{0}^\frac{t}{2} (t - \tau)^{- \frac{1}{2}\left(1-\frac{1}{p}\right)-\frac{\alpha-1}{2}-\frac{1}{2} } (1 + \tau)^{-1} d\tau 
        + C\int_{\frac{t}{2}}^{t} (t-\tau)^{-\frac{\alpha-1}{2}} \tau^{-\frac{1}{2}\left(1-\frac{1}{p}\right)-\frac{3}{2}} d\tau \nonumber\\
        &\le C t^{- \frac{1}{2}\left(1-\frac{1}{p}\right)-\frac{\alpha}{2}} \log(2 + t)
        + C t^{-\frac{1}{2}\left(1-\frac{1}{p}\right)-\frac{\alpha}{2}} \nonumber  \\
        &\le C t^{- \frac{1}{2}\left(1-\frac{1}{p}\right)-\frac{\alpha}{2}} \log(2 + t), \ \ t \ge2, \ \ 1<\alpha<3, \label{S-G-Linf}
    \end{align}
for any $1\le p \le \infty$. Therefore, we can conclude that \eqref{S-G-error} is true. 
\end{proof}

Next, we shall prove that the function $W(x, t)$ defined by \eqref{def.W} is well approximated by $\mathcal{M}\p_{x}G(x, t)$. 
Indeed, the following asymptotic relation holds true: 
\begin{prop}\label{prop.W-ap}
    Let $1<\alpha<3$ and $\beta \in \R$. Assume that $u_{0} \in H^{1}(\R) \cap L^{1}(\R)$, $xu_{0} \in L^{1}(\R)$ and $\| u_{0} \|_{H^{1}} + \| u_{0} \|_{L^{1}}$ is sufficiently small.
    Then, the solution $u(x, t)$ to \eqref{KdVB} satisfies
    \begin{equation}\label{W-ap}
        \lim_{t \to \infty} t^{\frac{1}{2}\left(1 - \frac{1}{p} \right) + \frac{1}{2}} \left\| W(t) - \mathcal{M}\p_{x}G(t) \right\|_{L^{p}} = 0,
    \end{equation}
    for any $1\le p \le \infty$, where $W(x, t)$, $G(x, t)$ and $\mathcal{M}$ are defined by \eqref{def.W}, \eqref{def.G-M} and \eqref{def.m-mathM}, respectively. 
\end{prop}
\begin{proof}
    Throughout this proof, for simplicity, we define the function $\rho(x, t)$ by 
    \begin{equation}\label{notations}
    \rho(x, t) := u^{3}(x, t) - (MG)^{3}(x, t)
    \end{equation}
    and the constants $\mathcal{M}_{0}$ and $\mathcal{M}_{1}$ by 
        \begin{equation}\label{notations-MM}
    \mathcal{M}_{0}:=\int_{0}^{1}\int_{\R}u^{3}(y, \tau)dyd\tau, \ \ \mathcal{M}_{1}:=\int_{1}^{\infty}\int_{\R} \rho(y, \tau) dyd\tau.
    \end{equation}
    Then, it follows from the definitions of $W(x, t)$ by \eqref{def.W} and $\mathcal{M}$ by \eqref{def.m-mathM} that 
    \begin{align}
        W(x, t) - \mathcal{M}\p_{x}G(x, t) 
            &= \int_{1}^{t}\p_{x}G(t-\tau)*\rho(\tau)d\tau - \mathcal{M}_{1}\p_{x}G(x, t)   \nonumber \\
            &\quad + \int_{0}^{1}\p_{x}G(t-\tau)*u^{3}(\tau)d\tau - \mathcal{M}_{0}\p_{x}G(x, t). \label{W-ap-pre}
    \end{align}
    
    In order to complete the proof of \eqref{W-ap}, we shall prove that the following asymptotic formulas are true: 
    \begin{align}
    & \lim_{t \to \infty} t^{\frac{1}{2}\left(1 - \frac{1}{p} \right) + \frac{1}{2} } \left\| \int_{1}^{t}\p_{x}G(t-\tau)*\rho(\tau)d\tau - \mathcal{M}_{1}\p_{x}G(t) \right\|_{L^{p}} = 0, \label{W-ap-1}\\
     & \lim_{t \to \infty} t^{\frac{1}{2}\left(1 - \frac{1}{p} \right) + \frac{1}{2} } \left\| \int_{0}^{1}\p_{x}G(t-\tau)*u^{3}(\tau)d\tau - \mathcal{M}_{0}\p_{x}G(t) \right\|_{L^{p}} = 0,  \label{W-ap-0}
    \end{align}
    for any $1\le p \le \infty$. In what follows, let us show only \eqref{W-ap-1} because we can prove \eqref{W-ap-0} in the same way. First, from the definition of $\mathcal{M}_{1}$ by \eqref{notations-MM}, we have  
    \begin{align} 
        &\int_{1}^{t}\p_{x}G(t-\tau)*\rho(\tau)d\tau - \mathcal{M}_{1}\p_{x}G(x, t) \nonumber \\
            &= \int_{1}^{t} \int_{\R} \p_{x} G(x - y, t - \tau) \rho(y, \tau) dy d\tau - \p_{x} G(x, t) \int_{1}^{\infty} \int_{\R} \rho(y, \tau) dy d\tau \nonumber \\
            &= \int_{1}^{t} \int_{\R} \p_{x}  \left( G(x - y, t - \tau) - G(x, t) \right) \rho(y, \tau) dy d\tau - \p_{x}G(x, t)\int_{t}^{\infty} \int_{\R} \rho(y, \tau) dy d\tau \nonumber \\
            &=: X(x, t) + Y(x, t). \label{split.X-Y}
    \end{align}
    In addition, we take small $\varepsilon > 0$ and then split the integral of $X(x, t)$ as follows:
    \begin{align*}
        X(x, t) &= \int_{\frac{\varepsilon t}{2}}^{t} \int_{\R} \p_{x} \left( G(x - y, t - \tau) - G(x, t) \right) \rho(y, \tau) dy d\tau \\
        &\quad + \int_{1}^\frac{\varepsilon t}{2} \left(\int_{|y| \ge \varepsilon \sqrt{t}} +\int_{|y| \le \varepsilon \sqrt{t}}\right)\p_{x} \left( G(x - y, t - \tau) - G(x, t)\right) \rho(y, \tau) dy d\tau \\
        &=: X_{1}(x, t) + X_{2}(x, t) + X_{3}(x, t).
    \end{align*}
        
        We shall evaluate $X_{1}(x, t)$, $X_{2}(x, t)$ and $X_{3}(x, t)$. In the following, let $1\le p \le \infty$. 
        From Proposition~\ref{prop.u-ap} and \eqref{notations}, we note that the following estimate holds:
         \begin{equation}\label{rho-est}
            \left\| \rho(\tau) \right\|_{L^{p}}
                       \le C\tau^{-\frac{1}{2}\left(1-\frac{1}{p}\right)-1} \left\{\tau^{-\frac{\min\{\alpha-1, 1\}}{2} }+\tau^{-\frac{1}{2}}\log(2 + \tau)\right\}, \ \ \tau \ge1.
    \end{equation}
    Now, let us evaluate $X_{1}(x, t)$. It follows from the Young inequality and \eqref{rho-est} that 
    \begin{align}
            \left\| X_{1}(t) \right\|_{L^{p}} 
            &\le \int_{\frac{\varepsilon t}{2}}^{t} \left\| \p_{x} G(t - \tau) \right\|_{L^{1}} \left\| \rho(\tau) \right\|_{L^{p}} d\tau + \left\| \p_{x} G(t) \right\|_{L^{p}} \int_{\frac{\varepsilon t}{2}}^{t} \left\| \rho(\tau) \right\|_{L^{1}} d\tau \nonumber \\
            &\le C\int_{\frac{\varepsilon t}{2}}^{t}  (t - \tau)^{-\frac{1}{2}} \tau^{-\frac{1}{2}\left(1-\frac{1}{p}\right)-1} \left\{\tau^{-\frac{\min\{\alpha-1, 1\}}{2} }+\tau^{-\frac{1}{2}}\log(2 + \tau)\right\} d\tau \nonumber \\
            &\quad + C t^{-\frac{1}{2}\left( 1 - \frac{1}{p} \right) - \frac{1}{2}} \int_{\frac{\varepsilon t}{2}}^{t} \tau^{-1} \left\{\tau^{-\frac{\min\{\alpha-1, 1\}}{2} }+\tau^{-\frac{1}{2}}\log(2 + \tau)\right\}d\tau \nonumber \\
            &\le C(\e)\, t^{-\frac{1}{2}\left(1-\frac{1}{p}\right) - \frac{1}{2}} \left\{t^{-\frac{\min\{\alpha-1, 1\}}{2} }+t^{-\frac{1}{2}}\log(2 + \tau)\right\}, \ \ t\ge \frac{2}{\varepsilon}.  \label{X1-est}
    \end{align}
    Next, we deal with $X_{2}(x, t)$. From Lemma~\ref{lem.G-est}, we can easily see that
    \begin{align} 
            \left\| X_{2}(t) \right\|_{L^{p}} 
            &\le \int_{1}^{\frac{\varepsilon t}{2}} \left( \left\| \p_{x} G(t - \tau) \right\|_{L^{p}} + \left\| \p_{x} G(t) \right\|_{L^{p}} \right)\int_{|y| \ge \varepsilon \sqrt{t}} | \rho(y, \tau) | dy d\tau \nonumber \\
            &\le C \int_{1}^{\frac{\varepsilon t}{2}} \left( (t-\tau)^{-\frac{1}{2}\left(1-\frac{1}{p}\right)-\frac{1}{2}} + t^{-\frac{1}{2}\left(1-\frac{1}{p}\right)-\frac{1}{2}} \right)\int_{|y| \ge \varepsilon \sqrt{t}} | \rho(y, \tau) | dy d\tau  \nonumber \\
            &\le C t^{-\frac{1}{2}\left(1-\frac{1}{p}\right)-\frac{1}{2}} Z(t), \ \ t\ge \frac{2}{\varepsilon}, \label{X2-est}
    \end{align}
    where we have defined
    \begin{equation*}
        Z(t) := \int_{1}^{\infty} \int_{|y| \geq \varepsilon \sqrt{t}} | \rho(y, \tau) | dy d\tau.
    \end{equation*}
    In addition, we note that $\int_{1}^{\infty} \int_{\R} |\rho(y, \tau)| dy d\tau < \infty$ by virtue of \eqref{rho-est}. 
    Therefore, applying the Lebesgue dominated convergence theorem, we obtain 
    \begin{equation}\label{Z-lim}
        \lim_{t \to \infty} Z(t) = \int_{1}^{\infty} \lim_{t \to \infty} \int_{|y| \geq \varepsilon \sqrt{t}} | \rho(y, \tau) | dy d\tau=0. 
    \end{equation}
    
    Finally, we shall treat $X_{3}(x, t)$.
    If $1 \le \tau \le \frac{\varepsilon t}{2}$ and $|y| \le \varepsilon \sqrt{t}$, by using the mean value theorem, the fact $\p_{t}G(x, t)=\p_{x}^{2}G(x, t)$ and Lemma \ref{lem.G-est}, we obtain
    \begin{align*} 
            &\left\| \p_{x} \left( G(\cdot - y, t - \tau) - G(\cdot, t) \right) \right\|_{L^{p}} \\
            &\le \left\| \p_{x} \left( G(\cdot - y, t - \tau) - G(\cdot, t-\tau) \right) \right\|_{L^{p}} + \left\| \p_{x} \left( G(\cdot, t - \tau) - G(\cdot, t) \right) \right\|_{L^{p}} \\
            &\le C (t - \tau)^{-\frac{1}{2}\left( 1 - \frac{1}{p} \right) - 1 } |y| + C (t - \tau)^{-\frac{1}{2}\left( 1 - \frac{1}{p} \right) - \frac{3}{2} }\tau \\
            &\le C \varepsilon t^{-\frac{1}{2}\left(1-\frac{1}{p}\right)-\frac{1}{2}}, \ \ t\ge \frac{2}{\varepsilon}.  
    \end{align*}
    Thus, combining the fact $\int_{1}^{\infty} \int_{\R} |\rho(y, \tau)| dy d\tau < \infty$ and the above estimate, we have
    \begin{align} 
            \left\| X_{3}(t) \right\|_{L^{p}} 
            &\le \int_{1}^{\frac{\varepsilon t}{2}} \int_{|y| \le \varepsilon \sqrt{t}} \left\| \p_{x} \left( G(\cdot - y, t - \tau) - G(\cdot, t) \right) \right\|_{L^{p}} |\rho(y, \tau)| dy d\tau \nonumber \\
            &\le C \varepsilon t^{-\frac{1}{2}\left(1-\frac{1}{p}\right)-\frac{1}{2}}, \ \ t\ge \frac{2}{\varepsilon}. \label{X3-est}
    \end{align}
    
    For $Y(x, t)$ in \eqref{split.X-Y}, using Lemma~\ref{lem.G-est} and \eqref{rho-est}, we obtain
    \begin{align}
            \left\| Y(t) \right\|_{L^{p}} 
            &\le \left\| \p_{x} G(t) \right\|_{L^p} \int_{t}^{\infty} \left\|\rho(\tau)\right\|_{L^{1}} d\tau \nonumber \\
            &\le C t^{-\frac{1}{2}\left( 1 - \frac{1}{p} \right) - \frac{1}{2}} \int_{t}^{\infty} \left\{\tau^{-\frac{\min\{\alpha-1, 1\}}{2}-1 }+\tau^{-\frac{3}{2}}\log(2 + \tau)\right\} d\tau \nonumber\\
            &=C t^{-\frac{1}{2}\left( 1 - \frac{1}{p} \right) - \frac{1}{2}} \left[-\frac{2}{\min\{\alpha-1, 1\}}\tau^{-\frac{\min\{\alpha-1, 1\}}{2}}\right]_{t}^{\infty} \nonumber\\
            &\quad +C t^{-\frac{1}{2}\left( 1 - \frac{1}{p} \right) - \frac{1}{2}} \left\{ \left[-2\tau^{-\frac{1}{2}}\log(2+\tau)\right]_{t}^{\infty}+2\int_{t}^{\infty}\tau^{-\frac{1}{2}}(2+\tau)^{-1}d\tau \right\} \nonumber\\
            &\le C t^{-\frac{1}{2}\left( 1 - \frac{1}{p} \right) - \frac{1}{2}} \left\{ t^{-\frac{\min\{\alpha-1, 1\}}{2}} + t^{-\frac{1}{2}}\log(2+t) \right\} , \ \ t>1. \label{Y-est}
    \end{align}
    
    Eventually, summing up \eqref{split.X-Y} and \eqref{X1-est} through \eqref{Y-est}, we arrive at
    \begin{align*}
        \limsup_{t \to \infty} t^{\frac{1}{2}\left(1-\frac{1}{p}\right)+\frac{1}{2}} \left\| \int_{1}^{t} \p_{x} G(t - \tau) * \rho (\tau) d\tau - \mathcal{M}_{1} \p_{x} G(t) \right\|_{L^{p}} \le C \varepsilon.
    \end{align*}
    Therefore, we finally obtain \eqref{W-ap-1} because $\varepsilon>0$ can be chosen arbitrarily small. 
    As a conclusion, combining \eqref{W-ap-pre}, \eqref{W-ap-1} and \eqref{W-ap-0}, we can see that \eqref{W-ap} is true. 
\end{proof}

Finally, in order to complete the proof of Theorem \ref{thm.Duhamel}, we shall prove that the leading term of $v(x, t)-V(x, t)$ is given by $\Psi(x, t)$ defined in \eqref{def.Psi}. 
Actually, we can show the following asymptotic formula: 
\begin{prop} \label{prop.v-V-Psi}
    Let $l \in \mathbb{N}\cup\{0\}$.
    Then, for any $1\le p \le \infty$, we have 
    \begin{align} \label{v-V-Psi.est}
        \left\| \p_{x}^{l} \left( v- V- \Psi \right) (t)\right\|_{L^{p}} \le C\left\|yF_{*}\right\|_{L^{1}} t^{-\frac{1}{2}\left(1-\frac{1}{p}\right)-1-\frac{l}{2}}, \ \ t>1,
    \end{align}
    where $v(x, t)$ and $V(x, t)$ are defined by \eqref{def.v-V}, while $\Psi(x, t)$ and $F_{*}(y)$ are defined by \eqref{def.Psi} and \eqref{def.F}, respectively.
\end{prop}
\begin{proof}
    First, from the definition of $V(x, t)$ by \eqref{def.v-V}, we can see that this function is the solution to the following Cauchy problem: 
    \begin{align*}
        \begin{split}
            &V_{t} - V_{xx} = \frac{1}{4\sqrt{3}\pi}t^{-1} \p_{x} G(x, t), \ \ x \in \R, \ t > 1, \\
            &V(x, 1) = 0, \ \ x \in \R.
        \end{split}
    \end{align*}
    Therefore, we can rewrite $V(x, t)$ as follows:
    \begin{align*} 
        V(x, t) = \frac{1}{4\sqrt{3}\pi} \int_{1}^{t} G(t - \tau) * \left( \tau^{-1} \p_{x}G(\tau) \right) d\tau.
    \end{align*}
    Thus, it follows from \eqref{def.v-V} and \eqref{def.F} that 
    \begin{align} 
            &v(x, t) - V(x, t) = \int_{1}^{t} \p_{x} G(t - \tau) * \left( G^{3}(\tau) - \frac{1}{4\sqrt{3}\pi} \tau^{-1} G(\tau) \right) d\tau \nonumber \\
            &= \int_{1}^{t} \int_{\R} \p_{x} G(x - y, t - \tau) \left( \frac{1}{8\sqrt{\pi^3}} \tau^{-\frac{3}{2}} e^{-\frac{3y^2}{4\tau}} - \frac{1}{8\sqrt{3\pi^3}} \tau^{-\frac{3}{2}} e^{-\frac{y^2}{4\tau}} \right) dy d\tau \nonumber \\
            &= \p_{x} \left( \int_{1}^{t} \int_{\R} G(x - y, t - \tau) F(y, \tau) dy d\tau \right)
            =: \p_{x} K(x, t). \label{def.K}
    \end{align}
    
    Next, let us transform $K(x, t)$ in the above \eqref{def.K}, by using the scaling argument. 
    Actually, from \eqref{def.G-M} and \eqref{def.F}, we note that the following self-similar structures of $G(x, t)$ and $F(y, \tau)$ hold true: 
    \begin{equation}\label{scaling}
    G(x, t)=\lambda G(\lambda x, \lambda^{2}t), \ \ F(y, \tau)=\lambda^{3}F(\lambda y, \lambda^{2}\tau), \ \ \text{for} \ \ \lambda>0.
    \end{equation}
By virtue of the above properties, we can easily obtain
\[
G(x-y, t-\tau)=\frac{1}{\sqrt{t}}G\left(\frac{x-y}{\sqrt{t}}, 1-\frac{\tau}{t}\right), \ \ F(y, \tau)=t^{-\frac{3}{2}}F\left(\frac{y}{\sqrt{t}}, \frac{\tau}{t}\right). 
\] 
Therefore, using the above equations and the change of variables, we get 
    \begin{align}
            K(x, t) 
             &=t^{-2}\int_{1}^{t}\int_{\R}G\left(\frac{x-y}{\sqrt{t}}, 1-\frac{\tau}{t} \right)F\left(\frac{y}{\sqrt{t}}, \frac{\tau}{t}\right)dyd\tau \nonumber \\
                          &=t^{-\frac{1}{2}}\int_{\frac{1}{t}}^{1}\int_{\R}G\left(\frac{x}{\sqrt{t}}-z, 1-s \right)F(z, s)dzds \nonumber \\
             &= t^{-\frac{1}{2}} \int_{\frac{1}{t}}^{1} \left( G(1 - s) * F(s) \right)\left( \frac{x}{\sqrt{t}} \right) ds. \label{K-change}
    \end{align}
    Hence, combining \eqref{def.K} and \eqref{K-change}, we have
    \begin{align} \label{v-V-re}
        v(x, t) - V(x, t) = t^{-\frac{1}{2}} \p_{x} \left( \int_{\frac{1}{t}}^{1} \left( G(1 - s) * F(s) \right)\left( \frac{x}{\sqrt{t}} \right) ds \right).
    \end{align}
    On the other hand, from the definition of $\Psi(x, t)$ by \eqref{def.Psi}, we can rewrite it as 
    \begin{align} \label{Psi-re}
        \Psi(x, t) = t^{-\frac{1}{2}} \p_{x} \left( \int_{0}^{1} \left( G(1 - s) * F(s) \right)\left( \frac{x}{\sqrt{t}} \right) ds \right).
    \end{align}
     Thus, for any $1\le p \le \infty$, it follows from \eqref{v-V-re} and \eqref{Psi-re} that
    \begin{align} 
            \left\| \p_{x}^{l} \left( v- V - \Psi \right) (t)\right\|_{L^{p}}  
            &= t^{-\frac{1}{2}} \left\| \p_{x}^{l+1} \left( \int_0^{\frac{1}{t}} \left( G(1 - s) * F(s) \right)\left( \frac{\cdot}{\sqrt{t}} \right) ds \right) \right\|_{L^{p}} \nonumber \\
            &= t^{-1 + \frac{1}{2p} - \frac{l}{2}} \left\| \p_{x}^{l+1} \left( \int_0^{\frac{1}{t}} G(1 - s) * F(s) ds \right) \right\|_{L^{p}}. \label{v-V-Psi-pre1}
    \end{align}
    
    In what follows, we shall evaluate the right hand side of \eqref{v-V-Psi-pre1}.
    Now, we note that $\int_\R F_{*}(y) dy = 0$.
    Therefore, using the mean value theorem, \eqref{scaling}, the Young inequality and Lemma~\ref{lem.G-est}, we obtain
    \begin{align} 
        &\left\| \p_{x}^{l+1} \left( \int_0^{\frac{1}{t}} G(1 - s) * F(s) ds \right) \right\|_{L^{p}} \nonumber \\
        &=\left\|  \p_{x}^{l+1}\int_{0}^{\frac{1}{t}} \int_{\R} G(\cdot - y, 1 - s) s^{-\frac{3}{2}}F_{*} \left(\frac{y}{\sqrt{s}}\right) dy ds \right\|_{L^{p}} \nonumber \\
        &=\left\| \int_{0}^{\frac{1}{t}} \int_{\R} \left( \int_{0}^{1} \p_{x}^{l+2} G(\cdot - \eta y, 1 - s) d\eta \right) ys^{-\frac{3}{2}}F_{*}\left( \frac{y}{\sqrt{s}} \right) dy ds \right\|_{L^{p}} \nonumber \\
        &= \int_{0}^{\frac{1}{t}} s^{-1} \int_{0}^{1} \left\| \int_{\R} \p_{x}^{l+2}\left( \frac{1}{\eta}G\left(\frac{\cdot}{\eta} - y, \frac{1 - s}{\eta^{2}}\right) \right) \frac{y}{\sqrt{s}} F_{*}\left( \frac{y}{\sqrt{s}} \right) dy \right\|_{L^{p}}d\eta ds \nonumber \\
        &\le C \int_{0}^{\frac{1}{t}} s^{-1} \left(\int_{0}^{1} \eta^{-1} \eta^{-(l+2)}\eta^{\frac{1}{p}}\eta^{1-\frac{1}{p}+(l+2)}d\eta \right)(1-s)^{-\frac{1}{2}\left(1-\frac{1}{p}\right)-\frac{l+2}{2}}s^{\frac{1}{2}}\left\|yF_{*}\right\|_{L^{1}}ds \nonumber \\
        &\le C\left\|yF_{*}\right\|_{L^{1}} \int_{0}^{\frac{1}{t}}s^{-\frac{1}{2}}(1-s)^{-\frac{1}{2}\left(1-\frac{1}{p}\right)-\frac{l+2}{2}}ds 
                \le C \left\|yF_{*}\right\|_{L^{1}}t^{-\frac{1}{2}}, \ \ t>1, \label{v-V-Psi-pre2}
    \end{align}
    for any $1\le p \le \infty$. Finally, combining \eqref{v-V-Psi-pre1} and \eqref{v-V-Psi-pre2}, we can conclude that the desired estimate \eqref{v-V-Psi.est} is true.
\end{proof}

\begin{proof}[\rm{\bf{Proof of Theorem~\ref{thm.Duhamel}}}]
Applying all the Propositions \ref{prop.S-G-error}, \ref{prop.W-ap} and \ref{prop.v-V-Psi} to \eqref{formulation}, 
we can immediately see that the desired formula \eqref{I-ap} holds. This completes the proof. 
\end{proof}

\begin{proof}[\rm{\bf{Proof of Corollary~\ref{thm.Main}}}]
Combining \eqref{KdVB-IE} and the Theorems \ref{thm.Karch-L} and \ref{thm.Duhamel}, 
we are able to conclude that the formulas \eqref{main-1}, \eqref{main-2} and \eqref{main-3} are true. This completes the proof. 
\end{proof}

\begin{proof}[\rm{\bf{Proof of Corollary~\ref{cor.main}}}]
We shall give the proof only for the case {\rm (ii)}, because the cases {\rm (i)} and {\rm (iii)} are similar to {\rm (ii)} (actually, these cases are easier than them the case {\rm (ii)}). First, from the definition of the constant $C_{\dag}$ by \eqref{const-d}, we note that
\begin{align*}
t^{ \frac{1}{2}\left(1-\frac{1}{p}\right) + \frac{1}{2} } \left\| \left(m+ \frac{\beta \mathcal{M}}{3}\right)\p_{x}G(t) + \frac{\beta M^{3}}{3}\Psi(t) - \frac{M}{N!}(tD_{x}^{\alpha}\p_{x})^{N} G(t)\right\|_{L^{p}} 
= 
C_{\dag}. 
\end{align*}
Now, noticing $\alpha>1$ and using Lemma~\ref{lem.G-est}, we obtain  
\begin{align*}
t^{ \frac{1}{2}\left(1-\frac{1}{p}\right) + \frac{1}{2} }\sum_{k=1}^{l}\frac{t^{k}}{k!}\left\|(D_{x}^{\alpha}\p_{x})^{k}\p_{x}G(t)\right\|_{L^{p}} 
 \le C\sum_{k=1}^{l}t^{-\frac{k(\alpha-1)}{2}} \to 0
\end{align*}
as $t\to \infty$, for any $1\le p\le \infty$ and $l\in \mathbb{N}$. 
Therefore, it follows from \eqref{main-2} that 
\begin{align*}
&\biggl|\, t^{ \frac{1}{2}\left(1-\frac{1}{p}\right) + \frac{1}{2} }\biggl\| u(t)-M\sum_{k=0}^{N-1}\frac{t^{k}}{k!}(D_{x}^{\alpha}\p_{x})^{k}G(t)+\frac{\beta M^3}{12\sqrt{3}\pi} (\log t)\p_{x} G(t)\biggl\|_{L^{p}}-C_{\dag}\, \biggl| \\
&\le t^{ \frac{1}{2}\left(1-\frac{1}{p}\right) + \frac{1}{2} } \biggl\| u(t) - \sum_{k=0}^{N-1}\frac{t^{k}}{k!}(D_{x}^{\alpha}\p_{x})^{k}\left\{MG(t)-m\p_{x}G(t)\right\}  \\
&\quad - \frac{M}{N!}(tD_{x}^{\alpha}\p_{x})^{N} G(t) +\frac{\beta M^3}{12\sqrt{3}\pi} (\log t)\p_{x} G(t) + \frac{\beta \mathcal{M}}{3}\p_{x}G(t) + \frac{\beta M^{3}}{3}\Psi(t) \biggl\|_{L^{p}} \\
        &\quad + |m|t^{ \frac{1}{2}\left(1-\frac{1}{p}\right) + \frac{1}{2} }\sum_{k=1}^{N-1}\frac{t^{k}}{k!}\left\|(D_{x}^{\alpha}\p_{x})^{k}\p_{x}G(t)\right\|_{L^{p}} \to 0        
         \end{align*}
as $t\to \infty$, for any $1\le p\le \infty$ and $N\in \mathbb{N}$. 
We note that the last term in the right hand side of the above inequality does not appear when $N=1$. 
From the above result, we can conclude that the desired formula \eqref{cor-2} is true. This completes the proof.  
\end{proof}

\section*{Acknowledgment}

This study is supported by Grant-in-Aid for Young Scientists Research No.22K13939, Japan Society for the Promotion of Science. The authors also would like to thank the anonymous referee for helpful and valuable comments on the paper.



\vskip5pt
\begin{flushleft}
Ikki Fukuda\\
Division of Mathematics and Physics, \\
Faculty of Engineering, Shinshu University\\
4-17-1, Wakasato, Nagano, 380-8553, JAPAN\\
E-mail: i\_fukuda@shinshu-u.ac.jp

\vskip5pt
Yota Irino\\
School of Information Science, \\
Graduate School of Advanced Science and Technology, \\
Japan Advanced Institute of Science and Technology \\
1-1, Asahidai, Nomi, Ishikawa, 923-1211, JAPAN \\
E-mail: s2210019@jaist.ac.jp
\end{flushleft}

\end{document}